\documentclass[11pt,english,a4paper]{amsart}
\pdfoutput=1

\usepackage[utf8]{inputenc}
\usepackage[T1]{fontenc}
\usepackage{lmodern}
\usepackage{babel}
\usepackage{amsmath,amssymb,amsthm}
\usepackage[marginratio={1:1,1:1},totalwidth=400pt,totalheight=600pt]{geometry}
\usepackage{microtype}
\usepackage{hyperref}
\usepackage{url}
\hypersetup{hidelinks, pdftitle={Z-stable and selfless inclusions for the infinite braid group}, pdfauthor={Tron Omland}}

\theoremstyle{plain}
\newtheorem{theorem}{Theorem}[section]
\newtheorem{lemma}[theorem]{Lemma}
\newtheorem{proposition}[theorem]{Proposition}
\newtheorem{corollary}[theorem]{Corollary}
\newtheorem{theoremx}{Theorem}

\theoremstyle{definition}
\newtheorem{remark}[theorem]{Remark}
\newtheorem{example}[theorem]{Example}

\newcommand{\C}{\mathbb{C}}
\newcommand{\N}{\mathbb{N}}
\newcommand{\Z}{\mathbb{Z}}
\newcommand{\R}{\mathbb{R}}
\newcommand{\T}{\mathbb{T}}
\newcommand{\F}{\mathbb{F}}
\newcommand{\Cr}{C^*_r}
\newcommand{\Zst}{\mathcal{Z}}
\DeclareMathOperator{\Mod}{Mod}
\DeclareMathOperator{\Fix}{Fix}
\DeclareMathOperator{\aut}{Aut}

\DeclareMathOperator{\Hom}{Hom}

\title[$\Zst$-stable and selfless inclusions]{$\Zst$-stable and selfless inclusions for the infinite braid group}

\author[Omland]{Tron Omland}
\address{Norwegian National Security Authority (NSM) \and Department of Mathematics, University of Oslo, Norway}
\email{tron.omland@gmail.com}

\date{October 24, 2026}


\begin{document}

\begin{abstract}
Let $B_\infty$ be the braid group on infinitely many strands and $P_\infty$ its pure braid group. We show that $\Cr(P_\infty)\subseteq\Cr(B_\infty)$ and all its intermediate inclusions are $\Zst$-stable and completely selfless. This gives natural examples of simple nonnuclear $\Zst$-stable $C^*$-algebras. Further examples come from $\operatorname{SL}_\infty(\Z)$ and $\operatorname{GL}_\infty(\Z)$. We also show that $\Cr(\F_\infty)\subseteq\Cr(B_\infty)$ is completely selfless for the free group $\F_\infty$ of braids in which only the first strand moves. All results hold for twisted group $C^*$-algebras. In contrast, $\Zst$-stability fails for $B_n$ with $n$ finite and for groups that are not inner amenable.
\end{abstract}

\maketitle

\section{Introduction}\label{sec:intro}

For every $2\leq n\leq\infty$ the \emph{braid group $B_n$} is defined by generators $s_1,s_2,\dotsc,s_{n-1}$ subject to the relations
\[
\begin{gathered}
s_is_{i+1}s_i=s_{i+1}s_is_{i+1}\text{ for all }1\leq i\leq n-2, \\
s_is_j=s_js_i\text{ when }\lvert i-j\rvert\geq2.
\end{gathered}
\]
Sending $s_i$ to the transposition $(i,i+1)$ defines a surjection from $B_n$ onto the symmetric group $S_n$. Its kernel is the \emph{pure braid group $P_n$}. Here $S_\infty$ is the group of permutations of $\N$ that move only finitely many elements. For $m<n\leq\infty$ we consider $B_m$ as the subgroup of $B_n$ coming from the first $m$ strands. Then $B_\infty=\bigcup_nB_n$ and $P_\infty=\bigcup_nP_n$. The group $P_n$ is generated by the braids
\[
a_{i,j}=s_{j-1}\dotsm s_{i+1}s_i^2s_{i+1}^{-1}\dotsm s_{j-1}^{-1},\qquad 1\leq i<j\leq n.
\]
Both $B_\infty$ and $P_\infty$ are $C^*$-simple~\cite[Section~6]{Om20} and the inclusion $\Cr(P_\infty)\subseteq\Cr(B_\infty)$ is $C^*$-irreducible~\cite[Example~7.4]{BO23}. That is, all its intermediate $C^*$-algebras are simple. By~\cite[Theorem~6.2]{BO23}, which builds on the Galois correspondence of~\cite[Theorem~4.4]{CS19}, they are exactly the algebras $\Cr(\Gamma)$ with $P_\infty\leq\Gamma\leq B_\infty$. There are uncountably many such groups, since $S_\infty$ contains $\bigoplus_\N\Z/2\Z$.

Let $\Zst$ be the Jiang--Su algebra. An inclusion $D\subseteq A$ of separable unital $C^*$-algebras is \emph{$\Zst$-stable} if there is an isomorphism $\Phi\colon A\to A\otimes\Zst$ with $\Phi(D)=D\otimes\Zst$~\cite{Sa25}. Selfless inclusions~\cite{HKEPR} and completely selfless inclusions~\cite{Ari26} are recalled in Section~\ref{sec:selfless}. An \emph{intermediate inclusion} of $D\subseteq A$ is an inclusion $E\subseteq F$ of $C^*$-algebras with $D\subseteq E\subseteq F\subseteq A$. For $\Cr(P_\infty)\subseteq\Cr(B_\infty)$ these are the inclusions $\Cr(\Lambda)\subseteq\Cr(\Gamma)$ with $P_\infty\leq\Lambda\leq\Gamma\leq B_\infty$.

\begin{theoremx}\label{t:A}
The inclusion $\Cr(P_\infty)\subseteq\Cr(B_\infty)$ and all its intermediate inclusions and $C^*$-algebras are $\Zst$-stable and completely selfless. In particular each intermediate $C^*$-algebra is simple with a unique tracial state, and it has stable rank one and strict comparison.
\end{theoremx}

The group $P_\infty$ is not amenable. So the intermediate $C^*$-algebras are nonnuclear, and Theorem~\ref{t:A} gives natural examples of simple nonnuclear $\Zst$-stable $C^*$-algebras. For reduced group $C^*$-algebras $\Zst$-stability is harder to obtain than selflessness. It fails whenever the group is not inner amenable (Proposition~\ref{p:inner}), for instance for nonabelian free groups and for Thompson's groups $T$ and $V$. The group $B_\infty$ avoids this obstruction, since every finite set of braids commutes with all braids on strands far out.

The proof of $\Zst$-stability is short. The pure braids on three strands far out contain a free group of rank two, and $\Zst$ embeds unitally in $\Cr(\F_2)$~\cite[Proposition~4.2 and Remark~4.3]{ThW14}. This gives approximately central copies of $\Zst$ in $\Cr(P_\infty)$. Section~\ref{sec:criterion} gives the general criterion behind this. It also applies with $P_\infty$ replaced by its commutator subgroup (Corollary~\ref{c:braids}).

Let $\F_\infty\leq P_\infty$ be the free group on the generators $\{a_{1,j}\}_{j=2}^\infty$. It consists of the braids in which only the first strand moves. Moreover $P_\infty\simeq\F_\infty\rtimes_\alpha P_\infty$, where $\alpha$ is Artin's representation~\cite[Section~6]{Om20}.

\begin{theoremx}\label{t:B}
The inclusion $\Cr(\F_\infty)\subseteq\Cr(B_\infty)$ and all its intermediate inclusions and $C^*$-algebras are completely selfless.
\end{theoremx}

The main input is that $B_n$ and a large power of a pseudo-Anosov braid in $B_{n+1}$ generate their free product (Proposition~\ref{p:free}). This follows from~\cite[Proposition~2.1]{AD19}. Since these free products lie inside $B_\infty$, each finite stage gives an embedding of $\Cr(B_n*\Z)$ into $\Cr(B_\infty)$ itself, not only into an ultrapower. Such embeddings remain isometric after tensoring with any $C^*$-algebra, so the inclusion is even completely selfless (Lemma~\ref{l:free}). Since $\F_\infty\leq P_\infty$, Theorem~\ref{t:B} gives the complete selflessness in Theorem~\ref{t:A}. Selfless inclusions into exact $C^*$-algebras can be obtained from $\Zst$-stable inclusions with the relative Dixmier property~\cite[Theorem~2.9]{HKEPR}. This route is not available for Theorem~\ref{t:B}, because $\Cr(\F_\infty)$ is not $\Zst$-stable (Proposition~\ref{p:inner}).

Theorem~\ref{t:B} also shows that $\Cr(\F_\infty)\subseteq\Cr(B_\infty)$ is $C^*$-irreducible, although $\F_\infty$ is not normal in $B_\infty$ (Remark~\ref{r:irred}). Selfless inclusions also arise from reduced amalgamated free products~\cite{GKPT}, from commensurators of hyperbolic groups~\cite{BF26}, from reduced twisted crossed products~\cite{FKOP26} and from twisted group $C^*$-algebras of amenable groups~\cite{Om26}.

By Lemma~\ref{l:annular}, Theorem~\ref{t:B} has a version for Artin's representation $\alpha$ of $B_\infty$ on $\F_\infty$. For every $\Lambda\leq B_\infty$ the coefficient inclusion $\Cr(\F_\infty)\subseteq\Cr(\F_\infty)\rtimes^r_{\alpha}\Lambda$ is completely selfless (Corollary~\ref{c:artin}). The same holds for the twisted actions $\alpha^\varphi$ of~\cite{Om20}. Here $\Lambda$ is arbitrary, for instance $B_\infty$, $P_\infty$, or an amenable subgroup. The intermediate $C^*$-algebras of these coefficient inclusions are exactly the crossed products by subgroups of $\Lambda$ (Remark~\ref{r:galois}). For $\Lambda=B_\infty$ the crossed product by $\alpha$ is $\Cr(B_{1,\infty})$, where $B_{1,\infty}$ consists of the braids that fix the endpoint of the first strand. It is a maximal $C^*$-subalgebra of $\Cr(B_\infty)$ (Remark~\ref{r:maximal}).

Selflessness is known for many reduced group $C^*$-algebras~\cite{AGKEP,Oz}. For $\Cr(B_\infty)$ alone it also follows from~\cite{Oz} and~\cite[Theorem~4.1]{Ro25}, since Proposition~\ref{p:free} lets one write $B_\infty$ as an increasing union of nonelementary free products. These free products are not inner amenable~\cite{DGO}, so this route does not give $\Zst$-stability (Proposition~\ref{p:inner}).

All results hold for reduced twisted group $C^*$-algebras, also when the two-cocycle is only defined on a subgroup of $B_\infty$ that contains $P_\infty$ or $\F_\infty$ (Section~\ref{sec:twists}). For $\Zst$-stability the only extra step is to pass to commutators in the free centralizers. Some such step is needed in general (Example~\ref{ex:scalars}).

Section~\ref{sec:criterion} also records an obstruction. If $G$ is not inner amenable, then $\Cr(G,\sigma)$ is not $\Zst$-stable for any two-cocycle $\sigma$ (Proposition~\ref{p:inner}). A variant shows that $\Cr(B_n,\sigma)$ and the crossed products of~\cite{Om20} are not $\Zst$-stable for finite $n$, while for $n=\infty$ they are $\Zst$-stable and completely selfless (Corollaries~\ref{c:twisted} and~\ref{c:artin}). Finally Section~\ref{sec:stable} applies Theorem~\ref{t:criterion} to other increasing unions with block sums, such as $\operatorname{SL}_\infty(\Z)$. For $\operatorname{SL}_\infty(\Z)$ and $\operatorname{GL}_\infty(\Z)$ the reduced $C^*$-algebras are also completely selfless (Remark~\ref{r:stable}).

\section{\texorpdfstring{$\Zst$}{Z}-stability and inner amenability}\label{sec:criterion}

For a separable unital $C^*$-algebra $A$ put $A_\infty=\ell^\infty(\N,A)/c_0(\N,A)$. We view $A$ as the constant sequences in $A_\infty$. If $D\subseteq A$ is a unital inclusion, then $D_\infty\subseteq A_\infty$. Suppose there is a unital $*$-homomorphism $\Zst\to D_\infty\cap A'$. Then $D\subseteq A$ is $\Zst$-stable by~\cite[Theorem~4.4]{Sa25}, after passing to an ultrapower. The same map lands in $E_\infty\cap F'$ for every intermediate inclusion $E\subseteq F$. Hence every intermediate inclusion and every intermediate $C^*$-algebra is $\Zst$-stable by the same theorem.

For $X\subseteq G$ let $C_G(X)$ denote the centralizer of $X$ in $G$. The following criterion is a $C^*$-version of a standard proof of the McDuff property, where central sequences are found in subgroups far out. Here these subgroups are copies of $\F_2$, since $\Cr(\F_2)$ contains a unital copy of $\Zst$.

\begin{theorem}\label{t:criterion}
Let $H\leq G$ be countable groups. Suppose that $H\cap C_G(X)$ contains a copy of $\F_2$ for every finite set $X\subseteq G$. Then there is a unital embedding ${\Zst\to\Cr(H)_\infty\cap\Cr(G)'}$. Consequently $\Cr(H)\subseteq\Cr(G)$ and all its intermediate inclusions and $C^*$-algebras are $\Zst$-stable.
\end{theorem}

\begin{proof}
Choose finite sets $X_1\subseteq X_2\subseteq\dotsb$ with union $G$. Choose $L_k\leq H\cap C_G(X_k)$ with $L_k\simeq\F_2$. There are unital embeddings $\varphi_k\colon\Zst\to\Cr(L_k)\subseteq\Cr(H)$~\cite[Proposition~4.2 and Remark~4.3]{ThW14}. Each $\lambda(g)$ commutes with the image of $\varphi_k$ for large $k$. These elements span a dense subspace of $\Cr(G)$. Hence $\lVert[\varphi_k(c),a]\rVert\to0$ for all $c\in\Zst$ and $a\in\Cr(G)$. Thus the $\varphi_k$ induce a unital $*$-homomorphism ${\Zst\to\Cr(H)_\infty\cap\Cr(G)'}$, which is injective since $\Zst$ is simple.
\end{proof}

Note that $H$ need not be normal in $G$, and $\Cr(G)$ need not be simple or exact.

Let $\sigma$ be a two-cocycle on $G$ and let $\lambda_\sigma$ be the left regular $\sigma$-projective representation (Section~\ref{sec:twists}). The canonical trace $\tau$ on $\Cr(G,\sigma)$ gives the norm $\lVert x\rVert_2=\tau(x^*x)^{1/2}$. Each $x\in\Cr(G,\sigma)$ has Fourier coefficients $\hat x(g)=\tau(\lambda_\sigma(g)^*x)$, and $\lVert x\rVert_2^2=\sum_g\lvert\hat x(g)\rvert^2$. Recall that $G$ is \emph{inner amenable} if $\ell^\infty(G\setminus\{e\})$ has a mean that is invariant under conjugation.

\begin{proposition}\label{p:inner}
Let $C$ be a central subgroup of a countable group $G$ such that $G/C$ is not inner amenable. Let $\sigma$ be a two-cocycle on $G$ such that $\Cr(C,\sigma)$ is commutative. Then $\Cr(G,\sigma)$ is not $\Zst$-stable. In particular $\Cr(G,\sigma)$ is not $\Zst$-stable for any $\sigma$ when $G$ is not inner amenable.
\end{proposition}

\begin{proof}
Put $A=\Cr(G,\sigma)$ and $Q=G/C$, and suppose that $A$ is $\Zst$-stable. Since $\Zst\simeq\Zst^{\otimes\infty}$, there is a unital $*$-homomorphism $\varphi\colon\Zst\to A_\infty\cap A'$. Let $\omega$ be a free ultrafilter on $\N$. Then $(x_j)_j\mapsto\lim_\omega\tau(x_j)$ is a tracial state on $A_\infty$. Its composition with $\varphi$ is the unique tracial state of $\Zst$, which is faithful. Choose $a,b\in\Zst$ with $ab\neq ba$. Lift $\varphi(a)$ and $\varphi(b)$ to bounded sequences $(a_j)$ and $(b_j)$ in $A$. Then $\lVert[a_j,x]\rVert\to0$ and $\lVert[b_j,x]\rVert\to0$ for all $x\in A$, but $\lim_\omega\lVert[a_j,b_j]\rVert_2>0$.

Let $\mathbb{E}\colon A\to\Cr(C,\sigma)$ be the conditional expectation that keeps the Fourier coefficients on $C$. For $x\in A$ define $\xi_x\in\ell^2(Q)$ by $\xi_x(gC)=(\sum_{c\in C}\lvert\hat x(gc)\rvert^2)^{1/2}$. Then $\lVert x-\mathbb{E}(x)\rVert_2$ is the norm of the restriction of $\xi_x$ to $Q\setminus\{C\}$. The coefficient of $\lambda_\sigma(h)x\lambda_\sigma(h)^*$ at $hgh^{-1}$ has modulus $\lvert\hat x(g)\rvert$. Hence
\[
\lVert\lambda_\sigma(h)x\lambda_\sigma(h)^*-x\rVert_2\geq\lVert h\cdot\xi_x-\xi_x\rVert_2,
\]
where $(h\cdot\xi)(gC)=\xi(h^{-1}ghC)$. Since $Q$ is not inner amenable, there are $\kappa>0$ and ${h_1,\dotsc,h_m\in G}$ with $\sum_i\lVert h_i\cdot\eta-\eta\rVert_2\geq\kappa\lVert\eta\rVert_2$ for all $\eta\in\ell^2(Q\setminus\{C\})$, see~\cite{Ef75}. Hence $\kappa\lVert x-\mathbb{E}(x)\rVert_2\leq\sum_i\lVert\lambda_\sigma(h_i)x\lambda_\sigma(h_i)^*-x\rVert_2$ for all $x\in A$. So $\lVert a_j-\mathbb{E}(a_j)\rVert_2\to0$ and $\lVert b_j-\mathbb{E}(b_j)\rVert_2\to0$. Since $\Cr(C,\sigma)$ is commutative and $\mathbb{E}$ is contractive, $\lVert[a_j,b_j]\rVert_2\to0$. This is a contradiction.
\end{proof}

The central subgroup $C$ is needed for the braid groups $B_n$ with $n$ finite, which are inner amenable since their center is nontrivial (Section~\ref{sec:artin}). Proposition~\ref{p:inner} is a $C^*$-counterpart of the fact that $W^*(G,\sigma)$ is McDuff when $\Cr(G,\sigma)$ is $\Zst$-stable with a unique tracial state. For instance $\Cr(G,\sigma)$ is not $\Zst$-stable when $G$ is acylindrically hyperbolic with trivial finite radical, since such groups are not inner amenable~\cite{DGO}. This includes free groups of rank at least two. In particular $\Cr(\F_2)$ and $\Cr(\F_\infty)$ are not $\Zst$-stable.

On the other hand $\Cr(\F_2)$ has strict comparison~\cite{AGKEP}. So strict comparison does not imply $\Zst$-stability for simple nonnuclear $C^*$-algebras. For simple separable unital nuclear infinite-dimensional $C^*$-algebras the two properties are equivalent by the recent solution of the Toms--Winter conjecture~\cite{EHSSTW,Th26}. Moreover strict comparison implies selflessness for simple unital infinite-dimensional $C^*$-algebras with a unique tracial state~\cite{Ro26}. Thompson's groups $T$ and $V$ are not inner amenable~\cite{HO17}. So $\Cr(T)$ and $\Cr(V)$ are not $\Zst$-stable, see also~\cite{BG26}. However $\Cr(V)$ is completely selfless~\cite[Appendix~A]{Ari26}. For Thompson's group $F$, in contrast, $\Cr(F)$ is $\Zst$-stable and selfless~\cite{BG26}.

\section{Braid groups}\label{sec:braids}

We use standard facts about braid groups~\cite{KT08}. Let $\Delta_n=s_1(s_2s_1)\dotsm(s_{n-1}\dotsm s_2s_1)$. For $n\geq3$ the center of $B_n$ and of $P_n$ is generated by
\[
z_n=\Delta_n^2=(s_1\dotsm s_{n-1})^n=a_{1,2}(a_{1,3}a_{2,3})\dotsm(a_{1,n}a_{2,n}\dotsm a_{n-1,n}),
\]
see~\cite[(1.2)]{Om20}. The abelianization of $P_n$ is free abelian on the classes of the $a_{i,j}$. So the class of $z_n$ has coefficient one at each of them.

\begin{lemma}\label{l:center}
Let $n\geq3$. Then $Z(B_n)\cap B_{n-1}=\{e\}$.
\end{lemma}

\begin{proof}
Suppose $z_n^l\in B_{n-1}$. Then $z_n^l\in P_n\cap B_{n-1}=P_{n-1}$. So its class has coefficient zero at $a_{1,n}$. But this coefficient is $l$, so $l=0$.
\end{proof}

\begin{lemma}\label{l:tails}
Let $k\geq1$. Then $T_k=\langle s_{k+1}^2,s_{k+2}^2\rangle\leq P_\infty$ is free of rank two, and it commutes with $B_k$.
\end{lemma}

\begin{proof}
The generators $s_{k+1}$ and $s_{k+2}$ commute with $s_1,\dotsc,s_{k-1}$. Conjugation by $(s_1\dotsm s_{k+2})^k$ maps $s_1$ and $s_2$ to $s_{k+1}$ and $s_{k+2}$~\cite{KT08}. In $B_3/Z(B_3)\simeq\operatorname{PSL}_2(\Z)$ the elements $s_1^2$ and $s_2^2$ map to the classes of
\[
\begin{pmatrix}1&2\\0&1\end{pmatrix}
\quad\text{and}\quad
\begin{pmatrix}1&0\\-2&1\end{pmatrix}.
\]
These generate a free group of rank two. Hence $T_k$ is free on $s_{k+1}^2$ and $s_{k+2}^2$.
\end{proof}

\begin{corollary}\label{c:braids}
The inclusion $\Cr([P_\infty,P_\infty])\subseteq\Cr(B_\infty)$ and all its intermediate inclusions and $C^*$-algebras are $\Zst$-stable.
\end{corollary}

\begin{proof}
A finite subset of $B_\infty$ lies in some $B_k$. By Lemma~\ref{l:tails} the group $T_k\leq P_\infty$ commutes with $B_k$. The group $[T_k,T_k]\leq[P_\infty,P_\infty]$ is free of infinite rank. So it contains a copy of $\F_2$. Now apply Theorem~\ref{t:criterion}.
\end{proof}

For instance $\Cr([B_\infty,B_\infty])$ is $\Zst$-stable, although $[B_\infty,B_\infty]$ does not contain $P_\infty$.

\section{Free products and selfless inclusions}\label{sec:selfless}

A $C^*$-probability space $(A,\rho)$ is a unital $C^*$-algebra $A$ with a state $\rho$ whose GNS representation is faithful. Group $C^*$-algebras carry their canonical trace $\tau$. Let $\omega$ be a free ultrafilter on $\N$ and let $A^\omega$ be the norm ultrapower of $A$ with the state $\rho^\omega$. An inclusion $D\subseteq(A,\rho)$ is \emph{selfless} if there are a $C^*$-probability space $(C,\kappa)$ with $C\neq\C$ and a state-preserving $*$-homomorphism
\[
\Psi\colon(A,\rho)*(C,\kappa)\to(A^\omega,\rho^\omega)
\]
that is the diagonal embedding on $A$ and maps $C$ into $D^\omega$~\cite[Definition~2.3]{HKEPR}. Here $*$ is the reduced free product. For each $C^*$-algebra $M$, the map $\Psi$ induces a $*$-homomorphism from the algebraic tensor product $M\odot(A*C)$ to $(M\otimes A)^\omega$. It sends $b\otimes x$ to the class of $(b\otimes x_k)_k$ when $\Psi(x)$ is the class of $(x_k)_k$. The inclusion is \emph{completely selfless} if $\Psi$ can be chosen such that all these maps are continuous for the minimal tensor norm~\cite[Section~2.2]{Ari26}, compare~\cite[Section~6]{Oz}. The algebra $(A,\rho)$ is \emph{selfless}~\cite{Ro25} or \emph{completely selfless}~\cite[Section~6]{Oz} if the inclusion $A\subseteq(A,\rho)$ is selfless or completely selfless.

Every intermediate inclusion and every intermediate $C^*$-algebra of a selfless inclusion is selfless~\cite[Theorem~2.5(i)]{HKEPR}. The same holds for completely selfless inclusions~\cite[Section~2.2]{Ari26}. For $D\subseteq E\subseteq F\subseteq A$ one restricts $\Psi$ to the $C^*$-subalgebra generated by $F$ and $C$. In all our examples $\rho$ and $\kappa$ are traces. Then the free product trace is faithful, so this subalgebra is $(F,\rho)*(C,\kappa)$. The next lemma is a relative form of the embedding in~\cite[Theorem~1]{Oz}, in the easy case where $G$ itself contains the free products.

\begin{lemma}\label{l:free}
Let $H\leq G$ be countable groups. Suppose that for every finite set $X\subseteq G$ there is $g\in H$ of infinite order with $\langle X,g\rangle=\langle X\rangle*\langle g\rangle$. Then $\Cr(H)\subseteq\Cr(G)$ and all its intermediate inclusions and $C^*$-algebras are completely selfless.
\end{lemma}

\begin{proof}
Choose finite sets $X_k$ that increase to $G$ and elements $g_k\in H$ as above, and put $G_k=\langle X_k\rangle$. An alternating product of nontrivial elements of $G_k$ and $\langle g_k\rangle$ is nontrivial. So $\Cr(G_k)$ and $\Cr(\langle g_k\rangle)$ are free with respect to $\tau$. Let $C(\T)$ carry the Haar state, given by normalized Lebesgue measure. Let $u\in C(\T)$ be the identity function. Since $g_k$ has infinite order, there is a trace-preserving isomorphism $C(\T)\to\Cr(\langle g_k\rangle)$ that sends $u$ to $\lambda(g_k)$. This gives trace-preserving embeddings
\[
\Psi_k\colon\Cr(G_k)*C(\T)\to\Cr(G)
\]
that fix $\Cr(G_k)$ and send $u$ to $\lambda(g_k)$.

The algebras $\Cr(G_k)*C(\T)$ increase, and their union $U$ is dense in $\Cr(G)*C(\T)$. For $x\in U$ let $\Psi(x)$ be the class of $(\Psi_k(x))_k$ in $\Cr(G)^\omega$, with arbitrary entries for the finitely many $k$ with $x\notin\Cr(G_k)*C(\T)$. Then $\Psi$ is a trace-preserving $*$-homomorphism on $U$. It is isometric since $\lVert\Psi_k(x)\rVert=\lVert x\rVert$ for large $k$. Hence $\Psi$ extends to $\Cr(G)*C(\T)$. It is the diagonal embedding on $\Cr(G)$, and $\Psi(u)\in\Cr(H)^\omega$.

For complete selflessness let $M$ be a $C^*$-algebra and let $\Phi$ be the map from $M\odot(\Cr(G)*C(\T))$ to $(M\otimes\Cr(G))^\omega$ induced by $\Psi$. For $y\in M\odot U$ the norm of $\Phi(y)$ is the limit along $\omega$ of the norms of $(\mathrm{id}_M\otimes\Psi_k)(y)$. These maps are injective $*$-homomorphisms on $M\otimes(\Cr(G_k)*C(\T))$, so $\lVert\Phi(y)\rVert=\lVert y\rVert$. Now let $y=\sum_{i=1}^rb_i\otimes x_i$ be arbitrary. Choose $x_{i,j}\in U$ with $x_{i,j}\to x_i$ and put $y_j=\sum_ib_i\otimes x_{i,j}$. Then $y_j\to y$, and $\Phi(y_j)\to\Phi(y)$ since $\lVert\Phi(b\otimes x)\rVert\leq\lVert b\rVert\lVert x\rVert$. Hence $\lVert\Phi(y)\rVert=\lim_j\lVert y_j\rVert=\lVert y\rVert$. Intermediate inclusions are handled as above.
\end{proof}

To verify the hypothesis of Lemma~\ref{l:free} for $\F_\infty\leq B_\infty$ we let $B_{n+1}$ act on the curve graph of a punctured sphere. We first recall some facts about mapping class groups~\cite{FM}. Let $S$ be a sphere with finitely many punctures. The \emph{mapping class group} $\Mod(S)$ is the group of isotopy classes of orientation-preserving homeomorphisms of $S$. These homeomorphisms may permute the punctures. A simple closed curve in $S$ is \emph{essential} if each side of it contains at least two punctures. An element of $\Mod(S)$ is \emph{pseudo-Anosov} if it has infinite order and no nonzero power of it fixes an essential curve up to isotopy. By the classification of mapping classes this agrees with the usual definition~\cite[Chapter~13]{FM}.

The \emph{curve graph} of $S$ has the isotopy classes of essential simple closed curves as vertices. Two vertices are adjacent if the curves can be made disjoint. If $S$ has at least five punctures, then the curve graph is connected and hyperbolic, and $\Mod(S)$ acts on it by isometries~\cite{MM99}. A group of isometries is \emph{elliptic} if its orbits are bounded. An isometry $f$ is \emph{loxodromic} if $k\mapsto f^kx$ is a quasi-isometric embedding of $\Z$, and then $f$ has two fixed points on the boundary. An element of $\Mod(S)$ is loxodromic if and only if it is pseudo-Anosov~\cite{MM99}. Moreover the action is \emph{acylindrical}~\cite{Bow08}. That is, for every $\varepsilon>0$ there are $R$ and $N$ with the following property. If two points have distance at least $R$, then at most $N$ elements move both of them at most $\varepsilon$. A group is \emph{acylindrically hyperbolic} if it is not virtually cyclic and has an acylindrical action on a hyperbolic space with a loxodromic element.

The braid group $B_m$ is the mapping class group of the closed disk with $m$ punctures $p_1,\dotsc,p_m$, where homeomorphisms and isotopies fix the boundary pointwise~\cite[Chapter~9]{FM}. Here $s_i$ is the half-twist that exchanges $p_i$ and $p_{i+1}$. Let $m\geq3$ and let $S$ be the sphere with punctures $p_1,\dotsc,p_m,p_\infty$. Capping the boundary with a disk with puncture $p_\infty$ gives a homomorphism $q_m\colon B_m\to\Mod(S)$. Its kernel is $Z(B_m)$, and its image is the stabilizer of $p_\infty$~\cite[Chapters~3 and~9]{FM}. We call $\beta\in B_m$ \emph{pseudo-Anosov} if $q_m(\beta)$ is pseudo-Anosov.

\begin{proposition}\label{p:free}
Let $n\geq3$ and let $\beta\in B_{n+1}$ be pseudo-Anosov. Then $\langle B_n,\beta^N\rangle$ is the free product $B_n*\langle\beta^N\rangle$ for all large $N$. Such $\beta$ exist in $\F_n=\langle a_{1,2},\dotsc,a_{1,n+1}\rangle$.
\end{proposition}

\begin{proof}
Put $q=q_{n+1}$, $H=q(B_n)$, and $f=q(\beta)$. The map $q$ is injective on $B_n$ by Lemma~\ref{l:center}. Hence $H$ is infinite and torsion-free. The sphere $S$ has $n+2\geq5$ punctures. So $\Mod(S)$ acts acylindrically on its curve graph, and $f$ is loxodromic. The group $H$ fixes the curve around $p_1,\dotsc,p_n$. This curve is essential since both sides contain at least two punctures. Thus $H$ is elliptic.

Let $\delta$ be a hyperbolicity constant. Let $A_{10\delta}(f)$ be the set of points $x$ with $d(x,fx)\leq\inf_yd(y,fy)+10\delta$. This is the quasi-axis of $f$ used in~\cite{AD19}. It is $f$-invariant and lies within bounded distance of an orbit of $f$. Let $\Fix_{50\delta}(H)$ be the set of points that every $h\in H$ moves at most $50\delta$. By~\cite[Proposition~2.1]{AD19} it suffices to show that $\Fix_{50\delta}(H)\cap A_{10\delta}(f)$ is bounded and that no $h\neq e$ in $H$ preserves $A_{10\delta}(f)$. For the first condition, suppose that $\Fix_{50\delta}(H)$ contains two points at distance at least $R$. Then infinitely many elements of $H$ move both points at most $50\delta$. For large $R$ this contradicts acylindricity. Hence $\Fix_{50\delta}(H)$ is bounded, and so is its intersection with $A_{10\delta}(f)$. For the second condition, let $E(f)$ be the stabilizer of the set of the two fixed points of $f$ on the boundary. It is virtually cyclic~\cite{DGO}. Thus $\langle f\rangle$ has finite index in $E(f)$, and the elements of infinite order in $E(f)$ are loxodromic. The elements of $H$ are elliptic and $H$ is torsion-free. Hence $H\cap E(f)=\{e\}$. An element that preserves $A_{10\delta}(f)$ preserves the set of the two fixed points of $f$, so it lies in $E(f)$. Consequently no $h\neq e$ in $H$ preserves $A_{10\delta}(f)$. Now~\cite[Proposition~2.1]{AD19} gives $\langle H,f^N\rangle=H*\langle f^N\rangle$ for all large $N$. The natural map $B_n*\langle\beta^N\rangle\to B_{n+1}$ composed with $q$ is an isomorphism onto $H*\langle f^N\rangle$. So the natural map is injective.

It remains to find such $\beta$ in $\F_n$. Forgetting the first strand maps the braids in $B_{n+1}$ that fix the endpoint of the first strand onto the braid group of the other $n$ strands. Its kernel is $\F_n$~\cite[Corollary~1.23]{KT08}. Let $\bar S$ be the sphere obtained from $S$ by filling in $p_1$. Moving $p_1$ once along a loop in $\bar S$ gives a mapping class of $S$, called a point push~\cite[Chapter~4]{FM}. Under $q$ the group $\F_n$ becomes the group of point pushes of $p_1$. A loop in $\bar S$ is \emph{filling} if it cannot be homotoped off any essential simple closed curve in $\bar S$. Since $\bar S$ has $n+1\geq4$ punctures, such loops exist. The point push of $p_1$ along a filling loop is pseudo-Anosov~\cite[Chapter~14]{FM}.
\end{proof}

\begin{proof}[Proof of Theorem~\ref{t:B}]
Let $X\subseteq B_\infty$ be finite. Choose $n\geq3$ such that $X\subseteq B_n$. Proposition~\ref{p:free} gives a pseudo-Anosov $\beta\in\F_n$ and $N$ such that $\langle B_n,g\rangle$ is the free product $B_n*\langle g\rangle$ for $g=\beta^N$. Then $g\in\F_\infty$, and $g$ has infinite order since $\beta$ has. Moreover $\langle X,g\rangle$ is the free product $\langle X\rangle*\langle g\rangle$. Therefore Lemma~\ref{l:free} applies.
\end{proof}

\begin{proof}[Proof of Theorem~\ref{t:A}]
Every intermediate inclusion of $\Cr(P_\infty)\subseteq\Cr(B_\infty)$ is also an intermediate inclusion of $\Cr([P_\infty,P_\infty])\subseteq\Cr(B_\infty)$ and of $\Cr(\F_\infty)\subseteq\Cr(B_\infty)$. So it is $\Zst$-stable by Corollary~\ref{c:braids} and completely selfless by Theorem~\ref{t:B}. The last statement follows from~\cite[Theorem~3.1]{Ro25}.
\end{proof}

\begin{remark}\label{r:irred}
The subgroup $\F_\infty$ is normal in $P_\infty$, but not in $B_\infty$. For instance $s_1a_{1,3}s_1^{-1}=a_{2,3}\notin\F_\infty$. Nevertheless $\Cr(\F_\infty)\subseteq\Cr(B_\infty)$ is $C^*$-irreducible. Indeed every intermediate $C^*$-algebra is selfless by Theorem~\ref{t:B}, hence simple~\cite[Theorem~3.1]{Ro25}. Moreover the inclusion has the relative Dixmier property~\cite[Theorem~2.5(iv)]{HKEPR}. This complements~\cite[Example~7.4]{BO23}, where the subgroup $P_\infty$ is normal.
\end{remark}

\section{Twisted group \texorpdfstring{$C^*$}{C*}-algebras}\label{sec:twists}

Let $\sigma\colon G\times G\to\T$ be a normalized two-cocycle. Let $\lambda_\sigma$ be the left regular $\sigma$-projective representation of $G$ with $\lambda_\sigma(g)\lambda_\sigma(h)=\sigma(g,h)\lambda_\sigma(gh)$. We also write $\sigma$ for its restrictions to subgroups.

\begin{lemma}\label{l:scalars}
Let $h\in G$. For $t\in C_G(h)$ put $\chi_h(t)=\sigma(t,h)\overline{\sigma(h,t)}$. Then $\chi_h$ is a character on $C_G(h)$. Hence $\lambda_\sigma(t)$ commutes with $\lambda_\sigma(h)$ for all $t\in[C_G(h),C_G(h)]$.
\end{lemma}

\begin{proof}
For $t\in C_G(h)$ we have $\lambda_\sigma(t)\lambda_\sigma(h)=\chi_h(t)\lambda_\sigma(h)\lambda_\sigma(t)$. Applying this to $t_1$, $t_2$, and $t_1t_2$ gives $\chi_h(t_1t_2)=\chi_h(t_1)\chi_h(t_2)$.
\end{proof}

\begin{theorem}\label{t:twisted}
Let $H\leq G$ be countable groups and let $\sigma$ be a two-cocycle on $G$.
\begin{itemize}
\item[(a)] Under the hypothesis of Theorem~\ref{t:criterion} the inclusion $\Cr(H,\sigma)\subseteq\Cr(G,\sigma)$ and all its intermediate inclusions and $C^*$-algebras are $\Zst$-stable.
\item[(b)] Under the hypothesis of Lemma~\ref{l:free} the inclusion $\Cr(H,\sigma)\subseteq\Cr(G,\sigma)$ and all its intermediate inclusions and $C^*$-algebras are completely selfless.
\end{itemize}
\end{theorem}

\begin{proof}
(a) Let $X_k$ and $L_k$ be as in the proof of Theorem~\ref{t:criterion}. The commutator subgroup $[L_k,L_k]$ is free of infinite rank, so we may replace $L_k$ by a copy of $\F_2$ in $[L_k,L_k]$. Then $\lambda_\sigma(t)$ commutes with $\lambda_\sigma(g)$ for all $t\in L_k$ and $g\in X_k$ by Lemma~\ref{l:scalars}. Every two-cocycle on a free group is a coboundary. So $\Cr(L_k,\sigma)\simeq\Cr(\F_2)$, and the rest of the proof of Theorem~\ref{t:criterion} applies.

(b) Let $G_k$ and $g_k$ be as in the proof of Lemma~\ref{l:free}. Let $x_1,\dotsc,x_r$ be nontrivial elements taken alternately from $G_k$ and $\langle g_k\rangle$. Then $\lambda_\sigma(x_1)\dotsm\lambda_\sigma(x_r)$ is a scalar multiple of $\lambda_\sigma(x_1\dotsm x_r)$, and $x_1\dotsm x_r\neq e$. So its trace is zero. Hence $\Cr(G_k,\sigma)$ and $\Cr(\langle g_k\rangle,\sigma)$ are free with respect to $\tau$. Moreover $\lambda_\sigma(g_k)^n$ is a scalar multiple of $\lambda_\sigma(g_k^n)$, so its trace is zero for $n\neq0$. So $u\mapsto\lambda_\sigma(g_k)$ defines a trace-preserving isomorphism $C(\T)\to\Cr(\langle g_k\rangle,\sigma)$. Now the proof of Lemma~\ref{l:free} applies.
\end{proof}

\begin{corollary}\label{c:twisted}
Let $\Gamma\leq B_\infty$ and let $\sigma$ be a two-cocycle on $\Gamma$, which need not extend to $B_\infty$ (Example~\ref{ex:scalars}). If $[P_\infty,P_\infty]\leq\Gamma$, then $\Cr([P_\infty,P_\infty],\sigma)\subseteq\Cr(\Gamma,\sigma)$ is $\Zst$-stable. If $\F_\infty\leq\Gamma$, then $\Cr(\F_\infty,\sigma)\subseteq\Cr(\Gamma,\sigma)$ is completely selfless. In both cases every intermediate inclusion and every intermediate $C^*$-algebra has the same property. In particular $\Cr(P_\infty,\sigma)\subseteq\Cr(\Gamma,\sigma)$ is $\Zst$-stable and completely selfless when $P_\infty\leq\Gamma$. Its intermediate $C^*$-algebras are then exactly the algebras $\Cr(\Lambda,\sigma)$ with $P_\infty\leq\Lambda\leq\Gamma$.
\end{corollary}

\begin{proof}
The proofs of Corollary~\ref{c:braids} and Theorem~\ref{t:B} apply with $B_\infty$ replaced by $\Gamma$, and with Theorem~\ref{t:twisted} in place of Theorem~\ref{t:criterion} and Lemma~\ref{l:free}. For the last part note that $P_\infty$ is normal in $\Gamma$ and $C^*$-simple. Moreover every nontrivial element of $B_\infty$ has an infinite $P_\infty$-conjugacy class~\cite[Section~6]{Om20}. Hence $\Cr(P_\infty,\sigma)$ is simple~\cite[Corollary~4.5]{BK}, and $(\Gamma,P_\infty,\sigma)$ satisfies the relative Kleppner condition. Now~\cite[Theorem~6.2]{BO23} gives the claim.
\end{proof}

\begin{example}\label{ex:scalars}
The groups $T_k$ of Lemma~\ref{l:tails} do not suffice in the twisted case. This is why the proof of Theorem~\ref{t:twisted}(a) passes to commutators. The abelianization of $P_\infty$ is free abelian on the classes of the $a_{i,j}$. Let $\ell_{i,j}\colon P_\infty\to\Z$ be the corresponding coordinate maps. Put $\mu=\ell_{1,2}$ and $\nu=\sum_{m\geq3}\ell_{m,m+1}$. The sum is finite on each braid. For $\theta\in\R\setminus\Z$ the bicharacter $\sigma(g,h)=e^{2\pi i\theta\mu(g)\nu(h)}$ is a two-cocycle on $P_\infty$. For $m\geq3$ the braids $a_{1,2}$ and $a_{m,m+1}$ commute. But
\[
\lambda_\sigma(a_{1,2})\lambda_\sigma(a_{m,m+1})=e^{2\pi i\theta}\lambda_\sigma(a_{m,m+1})\lambda_\sigma(a_{1,2}).
\]
So $\lambda_\sigma(a_{1,2})$ does not commute with $\lambda_\sigma(a_{m,m+1})$, however large $m$ is. In particular $\sigma$ is not a coboundary. If $2\theta\notin\Z$, then $\sigma$ is not cohomologous to the restriction of a two-cocycle on $B_\infty$. Indeed conjugation by $\Delta_{m+1}$ interchanges $a_{1,2}=s_1^2$ and $a_{m,m+1}=s_m^2$~\cite{KT08}. For a two-cocycle on $B_\infty$ the scalar $e^{2\pi i\theta}$ above would then equal its inverse.
\end{example}

\section{Artin's representation}\label{sec:artin}

We consider Artin's representation $\alpha$ of $B_\infty$ on the free group $\F_\infty$ with generators $x_1,x_2,\dotsc$~\cite[Section~6]{Om20}. Let $B_{1,\infty}\leq B_\infty$ be the group of braids that fix the endpoint of the first strand. It is the analogue of the annular braid group in~\cite{Om20}, with the order of the strands reversed.

\begin{lemma}\label{l:annular}
The assignment $a_{1,j}\mapsto x_{j-1}$ and $s_k\mapsto s_{k-1}$ for $j,k\geq2$ defines an isomorphism $B_{1,\infty}\simeq\F_\infty\rtimes_\alpha B_\infty$. Its restriction to $P_\infty$ is the isomorphism $P_\infty\simeq\F_\infty\rtimes_\alpha P_\infty$ of~\cite[Section~6]{Om20}.
\end{lemma}

\begin{proof}
Forgetting the first strand maps $B_{1,\infty}$ onto the braid group of the strands $2,3,\dotsc$ with kernel $\langle a_{1,j}:j\geq2\rangle$, which is free on the $a_{1,j}$~\cite[Corollary~1.23]{KT08}. The subgroup $\langle s_2,s_3,\dotsc\rangle\simeq B_\infty$ is a section. For $k\geq2$ we have
\[
s_ka_{1,k}s_k^{-1}=a_{1,k+1},\qquad s_ka_{1,k+1}s_k^{-1}=a_{1,k+1}^{-1}a_{1,k}a_{1,k+1},
\]
and $s_ka_{1,j}s_k^{-1}=a_{1,j}$ for $j\notin\{k,k+1\}$. The first identity is the definition of $a_{1,k+1}$. Since $s_k^2=a_{k,k+1}$, the second is equivalent to ${a_{1,k}a_{1,k+1}a_{k,k+1}=a_{1,k+1}a_{k,k+1}a_{1,k}}$, which holds by~\cite[Lemma~1.8.2]{Bi74}. For $j<k$ the third is clear. For $j\geq k+2$ write $a_{1,j}=ws_1^2w^{-1}$ with $w=s_{j-1}\dotsm s_2$. Then $s_kw=ws_{k+1}$, and $s_{k+1}$ commutes with $s_1$. These identities are~\cite[(1.1)]{Om20} for $\alpha(s_{k-1})$ after the substitution $x_{j-1}=a_{1,j}$.
\end{proof}

\begin{corollary}\label{c:artin}
Let $\Lambda\leq B_\infty$. Then $\Cr(\F_\infty)\subseteq\Cr(\F_\infty)\rtimes^r_\alpha\Lambda$ is a completely selfless inclusion. If $P_\infty\leq\Lambda$, then $\Cr(\F_\infty)\rtimes^r_\alpha\Lambda$ is also $\Zst$-stable.
\end{corollary}

\begin{proof}
We have $\Cr(\F_\infty)\rtimes^r_\alpha\Lambda\simeq\Cr(\F_\infty\rtimes_\alpha\Lambda)$. By Lemma~\ref{l:annular} the group $\F_\infty\rtimes_\alpha\Lambda$ is isomorphic to a subgroup $\Gamma$ of $B_\infty$ with $\F_\infty\leq\Gamma$. If $P_\infty\leq\Lambda$, then also $P_\infty\leq\Gamma$. Now apply Theorems~\ref{t:B} and~\ref{t:A}.
\end{proof}

The same holds for the twisted actions $\alpha^\varphi$ of~\cite{Om20}, where $\varphi\in Z^1(\Lambda,\Hom(\F_\infty,\T))$. Indeed $\Cr(\F_\infty)\rtimes^r_{\alpha^\varphi}\Lambda\simeq\Cr(\F_\infty\rtimes_\alpha\Lambda,\sigma^\varphi)$~\cite[Section~2.2]{Om20}, and $\sigma^\varphi$ is trivial on $\F_\infty$. So Corollary~\ref{c:twisted} applies.

\begin{remark}\label{r:galois}
The inclusion in Corollary~\ref{c:artin} is $C^*$-irreducible, since its intermediate $C^*$-algebras are selfless and hence simple~\cite[Theorem~3.1]{Ro25}. So by the Galois correspondence of~\cite[Theorem~4.4]{CS19}, see~\cite[Theorem~5.2]{BO23}, every intermediate $C^*$-algebra is the crossed product of $\Cr(\F_\infty)$ by a unique subgroup of $\Lambda$. For $\Lambda=B_\infty$, Lemma~\ref{l:annular} shows that the intermediate $C^*$-algebras of $\Cr(\F_\infty)\subseteq\Cr(B_{1,\infty})$ are exactly the algebras $\Cr(\Gamma)$ with $\F_\infty\leq\Gamma\leq B_{1,\infty}$. We do not know whether the same holds for $\Cr(\F_\infty)\subseteq\Cr(B_\infty)$, where $\F_\infty$ is not normal (Remark~\ref{r:irred}).
\end{remark}

\begin{remark}\label{r:maximal}
The subgroup $B_{1,\infty}$ is maximal in $B_\infty$. Indeed let $\pi\colon B_\infty\to S_\infty$ be the quotient map. Then $B_{1,\infty}$ is the preimage of the stabilizer of $1$. Let $\Gamma\leq B_\infty$ contain $B_{1,\infty}$ and some $\gamma\notin B_{1,\infty}$. Put $j=\pi(\gamma)(1)$, so that $j\neq1$. For every $m\notin\{1,j\}$ the transposition $(j\,m)$ lies in $\pi(B_{1,\infty})$, and $(j\,m)\pi(\gamma)$ maps $1$ to $m$. So $\pi(\Gamma)$ is transitive and contains the stabilizer of $1$. Hence $\pi(\Gamma)=S_\infty$. Since $\ker\pi=P_\infty\leq\Gamma$, we get $\Gamma=B_\infty$. By the Galois correspondence in Section~\ref{sec:intro}, $\Cr(B_{1,\infty})$ is thus a maximal $C^*$-subalgebra of $\Cr(B_\infty)$. By Theorem~\ref{t:A} the inclusion is also $\Zst$-stable and completely selfless.
\end{remark}

For finite $n$ the picture is different. For $n\geq3$ the group $B_n/Z(B_n)$ is $C^*$-simple~\cite[Remark~2.1]{Om20}. For $n=3$ it is $\operatorname{PSL}_2(\Z)$, which is virtually free. For $n\geq4$ it has finite index in the mapping class group of a sphere with $n+1\geq5$ punctures (Section~\ref{sec:selfless}). In both cases it is acylindrically hyperbolic with trivial finite radical, so it is not inner amenable~\cite{DGO}. The center $Z(B_n)$ is cyclic, so $\Cr(Z(B_n),\sigma)$ is commutative for every two-cocycle $\sigma$ on $B_n$. Hence $\Cr(B_n,\sigma)$ is not $\Zst$-stable by Proposition~\ref{p:inner}. The same argument applies to $G=\F_n\rtimes_\alpha K$ with $2\leq n<\infty$ and $K=B_n$ or $K=P_n$. Here $\F_n\rtimes_\alpha B_n$ is the annular braid group on $n+1$ strands and $\F_n\rtimes_\alpha P_n\simeq P_{n+1}$. The center $C$ of $G$ is infinite cyclic, and $G/C$ is a $C^*$-simple subgroup of finite index in $B_{n+1}/Z(B_{n+1})$~\cite[Section~1 and Remark~2.1]{Om20}. As above, $G/C$ is not inner amenable. So $\Cr(G,\sigma)$ is not $\Zst$-stable for any two-cocycle $\sigma$ on $G$. This includes the crossed products $\Cr(\F_n)\rtimes^r_{\alpha^\varphi}K\simeq\Cr(G,\sigma^\varphi)$ of~\cite[Section~2.2]{Om20}, also when they are simple with a unique tracial state, as they are for suitable $\varphi$~\cite[Theorems~3.2 and~3.4]{Om20}. So $\Zst$-stability appears only in the limit $n\to\infty$.

\section{Groups with block sums}\label{sec:stable}

Many families of groups $G_1\leq G_2\leq\dotsb$ come with block sums $G_m\times G_k\to G_{m+k}$ that restrict to the inclusion on $G_m$. Then $G_m$ commutes with a copy of $G_k$ in the union $G_\infty$. So Theorem~\ref{t:criterion} applies to $G_\infty$ as soon as some $G_k$ contains $\F_2$. The basic example is $G_k=\F_2^k$, where $\Cr(G_\infty)$ is the infinite tensor power of $\Cr(\F_2)$. The braid groups are another example. Artin's representation embeds $B_\infty$ in the group $\aut_f\F_\infty$ in (c) below. We list some more examples.

\begin{proposition}\label{p:stable}
The reduced group $C^*$-algebras of the following groups are $\Zst$-stable.
\begin{itemize}
\item[(a)] The Artin group of a Coxeter graph with countably many vertices, each of finite degree, and infinitely many edges.
\item[(b)] The increasing unions $\operatorname{SL}_\infty(\Z)$, $\operatorname{GL}_\infty(\Z)$ and $\operatorname{Sp}_\infty(\Z)$ of $\operatorname{SL}_n(\Z)$, $\operatorname{GL}_n(\Z)$ and $\operatorname{Sp}_{2n}(\Z)$ under the standard embeddings.
\item[(c)] The group $\aut_f\F_\infty$ of automorphisms of $\F_\infty=\langle x_1,x_2,\dotsc\rangle$ that fix all but finitely many $x_i$.
\end{itemize}
\end{proposition}

\begin{proof}
In each case every finite subset of the group lies in a subgroup that commutes with a copy of $\F_2$. Then Theorem~\ref{t:criterion} applies.

(a) Let $W$ be a finite set of vertices. Only finitely many edges meet $W$ or its neighbors. So we can choose an edge $\{u,v\}$ that meets neither. Then $s_u$ and $s_v$ commute with $s_w$ for all $w\in W$. Moreover $s_u^2$ and $s_v^2$ generate a free group of rank two~\cite{CP01}. For $B_\infty$ this is Lemma~\ref{l:tails}.

(b) The group $\operatorname{SL}_n(\Z)$ commutes with the block-diagonal copy of $\operatorname{SL}_2(\Z)$ in the next two coordinates, and $\operatorname{SL}_2(\Z)$ contains $\F_2$. The same works for $\operatorname{GL}$ and for $\operatorname{Sp}$, since $\operatorname{Sp}_2(\Z)=\operatorname{SL}_2(\Z)$.

(c) Each element of $\aut_f\F_\infty$ restricts to an automorphism of $\langle x_1,\dotsc,x_n\rangle$ for large $n$. Automorphisms of $\langle x_1,\dotsc,x_n\rangle$ commute with automorphisms of $\langle x_{n+1},x_{n+2}\rangle$, all extended by the identity. The inner automorphisms of $\langle x_{n+1},x_{n+2}\rangle$ form a free group of rank two.
\end{proof}

\begin{remark}\label{r:stable}
If a group $G$ as above is also $C^*$-simple and exact, then $\Cr(G)$ has a unique tracial state~\cite{BKKO}. So it is selfless by~\cite[Theorem~3]{Oz}. For instance $\operatorname{SL}_\infty(\Z)$ is the increasing union of the $C^*$-simple groups $\operatorname{SL}_{2k+1}(\Z)$~\cite{BCdlH}. So it is $C^*$-simple. Then $\operatorname{GL}_\infty(\Z)$ is $C^*$-simple by~\cite[Theorem~1.4]{BKKO}, since the centralizer of $\operatorname{SL}_\infty(\Z)$ in it is trivial. Both groups are exact as increasing unions of linear groups~\cite{BO08}. Hence $\Cr(\operatorname{SL}_\infty(\Z))$ and $\Cr(\operatorname{GL}_\infty(\Z))$ are simple, $\Zst$-stable, and selfless.

They are even completely selfless. Indeed both groups have the property $\mathrm{P}_{\mathrm{PHP}}$ of~\cite[Section~8]{Oz}, which implies complete selflessness~\cite[Theorem~14]{Oz}. A group has this property if every finite subset lies in a subgroup with this property. To see this, replace the sets $C_i\subseteq D_i$ in the definition by $C_iR\subseteq D_iR$, where $R$ is a set of representatives of the right cosets of that subgroup. Every finite subset of $\operatorname{SL}_\infty(\Z)$ lies in some $\operatorname{SL}_m(\Z)$ with $m$ odd, which is a Zariski-dense subgroup of $\operatorname{SL}_m(\R)=\operatorname{PSL}_m(\R)$. Every finite subset of $\operatorname{GL}_\infty(\Z)$ lies in a subgroup $\operatorname{GL}_n(\Z)\Gamma(3)$ of some $\operatorname{GL}_m(\Z)$ with $n<m$ and $m$ odd, where $\Gamma(3)$ is the kernel of $\operatorname{SL}_m(\Z)\to\operatorname{SL}_m(\Z/3\Z)$. This subgroup does not contain $-1$, so $g\mapsto\det(g)g$ embeds it as a Zariski-dense subgroup of $\operatorname{SL}_m(\R)$. Zariski-dense subgroups of $\operatorname{PSL}_m(\R)$ have property $\mathrm{P}_{\mathrm{PHP}}$ by~\cite[Proposition~16]{Oz}, compare~\cite{Vi}.
\end{remark}

\subsection*{AI statement}
The author asked an AI model whether $\Cr(B_\infty)$ is selfless or $\Zst$-stable. The results and exposition of this paper were then developed in further dialogue with various versions of Claude and ChatGPT. The author has verified and written their own presentation of all arguments and takes full responsibility for the content.


\begin{thebibliography}{99}

\bibitem{AD19}
C.~Abbott and F.~Dahmani.
\newblock Property {$P_{\mathrm{naive}}$} for acylindrically hyperbolic groups.
\newblock {\em Math. Z.}, 291:555--568, 2019.

\bibitem{AGKEP}
T.~Amrutam, D.~Gao, S.~{Kunnawalkam Elayavalli}, and G.~Patchell.
\newblock Strict comparison in reduced group {$C^*$}-algebras.
\newblock {\em Invent. Math.}, 242(3):639--657, 2025.

\bibitem{Ari26}
R.~Arimoto.
\newblock {$C^*$}-selflessness of groups of homeomorphisms.
\newblock Preprint, arXiv:2609.29399v2, 2026.

\bibitem{BF26}
A.~Basu and F.~Flores.
\newblock Selfless inclusions arising from commensurator groups of hyperbolic groups.
\newblock Preprint, arXiv:2605.12737, 2026.

\bibitem{BO23}
E.~B{\'e}dos and T.~Omland.
\newblock {$C^*$}-irreducibility for reduced twisted group {$C^*$}-algebras.
\newblock {\em J. Funct. Anal.}, 284(5):Paper No. 109795, 2023.

\bibitem{BCdlH}
M.~Bekka, M.~Cowling, and P.~de~la Harpe.
\newblock Some groups whose reduced {$C^*$}-algebra is simple.
\newblock {\em Publ. Math. IH\'ES}, 80:117--134, 1994.

\bibitem{Bi74}
J.~S. Birman.
\newblock {\em Braids, links, and mapping class groups}.
\newblock Ann. of Math. Stud. 82. Princeton Univ. Press, 1974.

\bibitem{Bow08}
B.~H. Bowditch.
\newblock Tight geodesics in the curve complex.
\newblock {\em Invent. Math.}, 171(2):281--300, 2008.

\bibitem{BKKO}
E.~Breuillard, M.~Kalantar, M.~Kennedy, and N.~Ozawa.
\newblock {$C^*$}-simplicity and the unique trace property for discrete groups.
\newblock {\em Publ. Math. IH\'ES}, 126:35--71, 2017.

\bibitem{BO08}
N.~P. Brown and N.~Ozawa.
\newblock {\em {$C^*$}-algebras and finite-dimensional approximations}.
\newblock Grad. Stud. Math. 88. Amer. Math. Soc., 2008.

\bibitem{BG26}
C.~Bruce and S.~Geffen.
\newblock {$C^*$}-selflessness and {$\mathcal{Z}$}-stability for {T}hompson's group {$F$}.
\newblock Preprint, arXiv:2610.28473, 2026.

\bibitem{BK}
R.~S. Bryder and M.~Kennedy.
\newblock Reduced twisted crossed products over {$C^*$}-simple groups.
\newblock {\em Int. Math. Res. Not. IMRN}, 2018(6):1638--1655, 2018.

\bibitem{CS19}
J.~Cameron and R.~R. Smith.
\newblock A {G}alois correspondence for reduced crossed products of unital simple {$C^*$}-algebras by discrete groups.
\newblock {\em Canad. J. Math.}, 71:1103--1125, 2019.
\newblock Corrigendum in {\em Canad. J. Math.}, 72:557--562, 2020.

\bibitem{CP01}
J.~Crisp and L.~Paris.
\newblock The solution to a conjecture of {T}its on the subgroup generated by the squares of the generators of an {A}rtin group.
\newblock {\em Invent. Math.}, 145(1):19--36, 2001.

\bibitem{DGO}
F.~Dahmani, V.~Guirardel, and D.~Osin.
\newblock Hyperbolically embedded subgroups and rotating families in groups acting on hyperbolic spaces.
\newblock {\em Mem. Amer. Math. Soc.}, 245(1156), 2017.

\bibitem{Ef75}
E.~G. Effros.
\newblock Property {$\Gamma$} and inner amenability.
\newblock {\em Proc. Amer. Math. Soc.}, 47:483--486, 1975.

\bibitem{EHSSTW}
S.~Evington, S.~Hua, C.~Schafhauser, A.~Seth, A.~Tikuisis, and S.~White.
\newblock Regular simple nuclear {$C^*$}-algebras.
\newblock Preprint, arXiv:2610.09390, 2026.

\bibitem{FM}
B.~Farb and D.~Margalit.
\newblock {\em A primer on mapping class groups}.
\newblock Princeton Math. Ser. 49. Princeton Univ. Press, 2012.

\bibitem{FKOP26}
F.~Flores, M.~Klisse, M.~{{\'O} Cobhthaigh}, and M.~Pagliero.
\newblock Pureness and stable rank one for reduced twisted group {$C^*$}-algebras of certain group extensions.
\newblock Preprint, arXiv:2601.19758, 2026.

\bibitem{GKPT}
D.~Gao, S.~{Kunnawalkam Elayavalli}, G.~Patchell, and L.~Teryoshin.
\newblock Selfless reduced amalgamated free products and {HNN} extensions.
\newblock Preprint, arXiv:2604.06982, 2026.

\bibitem{HO17}
U.~Haagerup and K.~K. Olesen.
\newblock Non-inner amenability of the {T}hompson groups {$T$} and {$V$}.
\newblock {\em J. Funct. Anal.}, 272(11):4838--4852, 2017.

\bibitem{HKEPR}
B.~Hayes, S.~{Kunnawalkam Elayavalli}, G.~Patchell, and L.~Robert.
\newblock Selfless inclusions of {$C^*$}-algebras.
\newblock Preprint, arXiv:2510.13398v2, 2025.

\bibitem{KT08}
C.~Kassel and V.~Turaev.
\newblock {\em Braid groups}.
\newblock Grad. Texts in Math. 247. Springer, 2008.

\bibitem{MM99}
H.~A. Masur and Y.~N. Minsky.
\newblock Geometry of the complex of curves. {I}. {H}yperbolicity.
\newblock {\em Invent. Math.}, 138(1):103--149, 1999.

\bibitem{Om20}
T.~Omland.
\newblock Dynamical systems and operator algebras associated to {A}rtin's representation of braid groups.
\newblock {\em J. Operator Theory}, 83(1):55--72, 2020.

\bibitem{Om26}
T.~Omland.
\newblock Selflessness for twisted group {$C^*$}-algebras of amenable groups and their inclusions.
\newblock Preprint, arXiv:2606.27198, 2026.

\bibitem{Oz}
N.~Ozawa.
\newblock Proximality and selflessness for group {$C^*$}-algebras.
\newblock Preprint, arXiv:2508.07938v8, 2026.

\bibitem{Ro25}
L.~Robert.
\newblock Selfless {$C^*$}-algebras.
\newblock {\em Adv. Math.}, 478:Paper No. 110409, 2025.

\bibitem{Ro26}
L.~Robert.
\newblock Strict comparison and selflessness.
\newblock Preprint, arXiv:2609.15776v2, 2026.

\bibitem{Sa25}
P.~Sarkowicz.
\newblock Tensorially absorbing inclusions of {$C^*$}-algebras.
\newblock {\em Canad. J. Math.}, 77(4):1315--1346, 2025.

\bibitem{Th26}
H.~Thiel.
\newblock Ranks of operators and pureness of {$C^*$}-algebras.
\newblock Preprint, arXiv:2610.06761, 2026.

\bibitem{ThW14}
H.~Thiel and W.~Winter.
\newblock The generator problem for {$\mathcal{Z}$}-stable {$C^*$}-algebras.
\newblock {\em Trans. Amer. Math. Soc.}, 366(5):2327--2343, 2014.

\bibitem{Vi}
I.~Vigdorovich.
\newblock Selfless reduced {$C^*$}-algebras of linear groups.
\newblock {\em Proc. Lond. Math. Soc. (3)}, 133(1):Paper No. e70180, 2026.

\end{thebibliography}
\end{document}